\documentclass{article}

\usepackage{amsmath,amssymb}
\usepackage{amsthm}
\usepackage{cite} %%Seems not compatible with elsarticle.cls
\usepackage{enumitem}
\usepackage{url}
\usepackage{multicol}

\usepackage{color}
\usepackage{harpoon}

\usepackage{tikz}
\usetikzlibrary{positioning}
\usetikzlibrary{arrows}
\tikzset{
every node/.style={circle, draw=black, inner sep=3pt},
every picture/.style=thick,
bluenode/.style={circle, draw=black, fill=blue!40, very thick, minimum size=3mm,inner sep=1mm},
whitenode/.style={circle, draw=black, fill=black!10, very thick, minimum size=3mm,inner sep=1mm},
squarednode/.style={rectangle, draw=red!60, fill=red!5, very thick, minimum size=3mm},
every loop/.style={min distance=8mm} %% Default loop has arrow.
}

\newtheorem{theorem}{Theorem}[section]
\newtheorem{lemma}[theorem]{Lemma}
\newtheorem{proposition}[theorem]{Proposition}
\newtheorem{corollary}[theorem]{Corollary}

\theoremstyle{definition}
\newtheorem{definition}[theorem]{Definition}
\newtheorem{observation}[theorem]{Observation}
\newtheorem{remark}[theorem]{Remark}
\newtheorem{example}[theorem]{Example}
\newtheorem{question}[theorem]{Question}

\newenvironment{thm}{\begin{theorem}}{\end{theorem}}
\newenvironment{lem}{\begin{lemma}}{\end{lemma}}
\newenvironment{prop}{\begin{proposition}}{\end{proposition}}
\newenvironment{cor}{\begin{corollary}}{\end{corollary}}

\newenvironment{defn}{\begin{definition}\bgroup\rm }{\egroup\end{definition}}

\newenvironment{rem}{\begin{remark}\bgroup\rm }{\egroup\end{remark}}
\newenvironment{ex}{\begin{example}\bgroup\rm }{\egroup\end{example}}

\newenvironment{enumarabic}{\begin{enumerate}[label={\rm (\arabic*)}]}{\end{enumerate}}
\def \mr {\operatorname{mr}}
\def \rank{\operatorname{rank}}
\def \nul{\operatorname{null}}
\def \cc{\operatorname{cc}}

\def \S{\mathcal{S}}

\def \dunion{\mathbin{\dot\cup}}
\newcommand{\OS}{\operatorname{OS}}

\newcommand{\Zelld}{Z_{\dot{\ell}}}
\newcommand{\Zell}{Z_{\ell}}
\newcommand{\Zminus}{Z_-}
\newcommand{\ZL}{Z_L}
\newcommand{\Zarrow}{\xrightarrow{Z}}
\newcommand{\Zelldarrow}{\xrightarrow{\Zelld}}
\newcommand{\Zminusarrow}{\xrightarrow{\Zminus}}
\newcommand{\ZLarrow}{\xrightarrow{\ZL}}
\newcommand{\Selld}{\S_{\dot{\ell}}}
\newcommand{\Szero}{\S_0}
\newcommand{\Melld}{M_{\dot{\ell}}}
\newcommand{\Mzero}{M_0}
\newcommand{\ML}{M_L}
\newcommand{\mrelld}{\mr_{\dot{\ell}}}
\newcommand{\mrzero}{\mr_0}
\newcommand{\mrL}{\mr_L}
\newcommand{\gd}{\gamma_{\rm gr}}
\newcommand{\gdZ}{\gd^Z}
\newcommand{\gdt}{\gd^t}
\newcommand{\gdL}{\gd^L}
\newcommand{\BL}{\mathcal{B}_L}
\newcommand{\gag}{\gamma_{\rm g}}

\begin{document}

%%TITLE
\title{Zero forcing number, Grundy domination number, and their variants}

\author{
        Jephian C.-H. Lin\footnotemark[2]
        }

%\date{June 25, 2015} % Lin
\date{\today}

\maketitle

\renewcommand{\thefootnote}{\fnsymbol{footnote}}
\footnotetext[2]{
Department of Mathematics, Iowa State University, 
Ames, IA 50011, USA (chlin@iastate.edu).}

\renewcommand{\thefootnote}{\arabic{footnote}}

\begin{abstract}
  This paper presents strong connections between four variants of the zero forcing number and four variants of the Grundy domination number.  These connections bridge the domination problem and the minimum rank problem.  We show that the Grundy domination type parameters are bounded above by the minimum rank type parameters.  We also give a method to calculate the $L$-Grundy domination number by the Grundy total domination number, giving some linear algebra bounds for the $L$-Grundy domination number.
\end{abstract}

\medskip 

\noindent{\bf AMS subject classifications:}
%05C38, %%Paths and cycles
05C50, %%Graphs and linear algebra (matrices, eigenvalues, etc.)
05C57, %%Games on graphs
05C69, %%Dominating sets, independent sets, cliques
05C70, %%Factorization, matching, partitioning, covering and packing
%05C75, %%Structural characterization of families of graphs
%05C83, %%Graph minors
05C85. %%Graph algorithms
%15A03, %%Vector spaces, linear dependence, rank
%15A18, %%Eigenvalues, singular values, and eigenvectors
%15A29. %%Inverse problems
%15B35, %%Sign pattern matrices
%15B57. %%Hermitian, skew-Hermitian, and related matrices
%58C15, %%Implicit function theorems; global Newton methods
%65F18. %%Inverse eigenvalue problems
%68R10. %%Graph theory (including graph drawing)

\section{Introduction}
The zero forcing number $Z(G)$ considers a propagation process on a simple graph and refers to the minimum number of blue vertices so that all vertices can turn blue eventually under some color-change rule.  The zero forcing number was introduced independently by mathematicians for studying the minimum rank problem \cite{AIM} and physicists for quantum control \cite{BG07}.  The zero forcing number has been studied extensively, and many variants have been introduced; see, e.g., \cite{HLA46,param} and the references therein.  Variants of zero forcing have applications to the fast-mixed search in computer science \cite{FMY16} and the cops-and-robber game in graph theory \cite{param}.  Zero forcing can also be used for designing logic circuits \cite{BGHSY15}.  

On the other hand, a \emph{dominating set} $X$ of a simple graph $G$ is a set of vertices such that every vertex outside of $X$ has a neighbor in $X$, and the \emph{domination number} $\gamma(G)$ is the minimum cardinality of a dominating set.  One way to find a minimum dominating set is by the greedy algorithm:  Start with an empty set $X$.  Find a vertex $v$ such that
\[N[v]\setminus\bigcup_{x\in X}N[x]\neq \emptyset\]
and add $v$ to $X$, where $N[x]$ is the set of closed neighbors of $x$.  Keep doing this step until no more vertex can be added to $X$.  When the algorithm ends, it means $\bigcup_{x\in X}N[x]=V(G)$ and $X$ is a dominating set.  At any stage of the greedy algorithm, the elements in $X$ can be written as a sequence $(v_1,\ldots,v_k)$ such that
\[N[v_i]\setminus\bigcup_{j=1}^{i-1}N[v_j]\neq\emptyset\]
for all $i=1,\ldots, k$.  This condition guarantees that the newly added vertex is not redundant.  Depending on the searching process of adding a new vertex, the greedy algorithm can either return a minimum dominating set or a larger set with the abovementioned structure.  The Grundy domination number $\gd(G)$ refers to the maximum length of a sequence given by the greedy algorithm.

The Grundy domination number $\gd(G)$ and the Grundy total domination number $\gdt(G)$ was introduced in \cite{DinGraph} and \cite{TDinGraph}, respectively.  Bre\v{s}ar et al.~\cite{BBGKKPTV17} introduced the other two variants of the Grundy domination number, namely, the $Z$-Grundy domination number $\gdZ(G)$ and the $L$-Grundy domination number; the authors also showed that $Z(G)+\gdZ(G)=|V(G)|$, so finding the value of $\gdZ(G)$ is equivalent to finding the value of $Z(G)$.

In this paper, we will show that not only the $Z$-Grundy domination number but every Grundy domination type parameters mentioned in \cite{BBGKKPTV17} relates to a zero forcing type parameter.  That is, we will show in Theorem~\ref{mainthm} that
\begin{multicols}{2}
\begin{enumarabic}
\item $Z(G)=n-\gdZ(G)$,
\item $\Zelld(G)=n-\gd(G)$,
\item $\Zminus(G)=n-\gdt(G)$, 
\item $\ZL(G)=n-\gdL(G)$,
\end{enumarabic}
\end{multicols}
\noindent for any graph $G$.  The definitions of these parameters will be introduced in Section~\ref{prelim}.

In Section~\ref{mainsec}, we will prove the relations between the zero forcing type parameters and the Grundy domination type parameters, and many inequalities are given.  Section~\ref{mrboundssec} includes the definitions of the minimum rank type parameters, and we will show that each Grundy domination type parameters is bounded above by a minimum rank type parameter.  Finally, in Section~\ref{lgboundsec} we will provide a way to calculate the $L$-Grundy domination number by the Grundy total domination number and give some linear algebra bounds for $\gdL(G)$.  Figure~\ref{allparam} illustrates all related parameters, where a line connecting two parameters means the lower parameter is bounded above by the upper parameter.

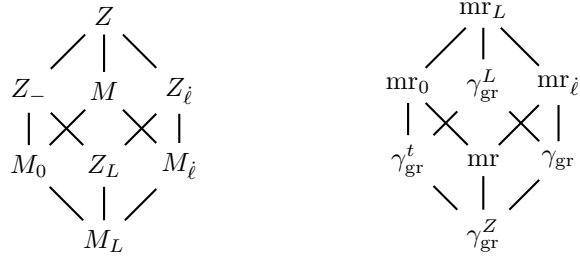
\begin{figure}[h]
\begin{center}
\begin{tikzpicture}[every node/.style={draw=none}]
\node (Z) at (1,3) {$Z$};
\node (Zminus) at (0,2) {$\Zminus$};
\node (Zelld) at (2,2) {$\Zelld$};
\node (ZL) at (1,1) {$\ZL$};
\node (M) at (1,2) {$M$};
\node (Mzero) at (0,1) {$\Mzero$};
\node (Melld) at (2,1) {$\Melld$};
\node (what) at (1,0) {$\ML$};
\draw (what) -- (Mzero) -- (Zminus) -- (Z);
\draw (what) -- (Melld) -- (Zelld) -- (Z);
\draw (Zminus) -- (ZL) -- (Zelld);
\draw (Mzero) -- (M) -- (Melld);
\draw (what) -- (ZL);
\draw (M) -- (Z);
\end{tikzpicture}
\hfil
\begin{tikzpicture}[every node/.style={draw=none}]
\node (what) at (1,3) {$\mrL$};
\node (mrzero) at (0,2) {$\mrzero$};
\node (mrelld) at (2,2) {$\mrelld$};
\node (mr) at (1,1) {$\mr$};
\node (gdL) at (1,2) {$\gdL$};
\node (gdt) at (0,1) {$\gdt$};
\node (gd) at (2,1) {$\gd$};
\node (gdZ) at (1,0) {$\gdZ$};
\draw (gdZ) -- (gdt) -- (mrzero) -- (what);
\draw (gdZ) -- (gd) -- (mrelld) -- (what);
\draw (mrzero) -- (mr) -- (mrelld);
\draw (gdt) -- (gdL) -- (gd);
\draw (gdZ) -- (mr);
\draw (gdL) -- (what);
\end{tikzpicture}
\end{center}
\caption{An illustration of related parameters}
\label{allparam}
\end{figure}

\subsection{Preliminaries}
\label{prelim}
All graph considered are simple and finite.  For a vertex $x$ of $G$, the \emph{open neighborhood} of $x$ in $G$ is denoted as $N_G(x)$, while the \emph{closed neighborhood} is denoted as $N_G[x]$.  When the context is clear, we will simply use $N(x)$ and $N[x]$.  

Let $G$ be a graph.  The \emph{zero forcing game} is a color-change game such that each vertex is colored blue or white initially, and then the \emph{color change rule} (CCR) is applied repeatedly.  The color change rule varies from different variants of the zero forcing game, as we will see in Definition~\ref{Zdefn}.  If starting with an initial blue set $B\subseteq V(G)$ makes every vertex turn blue eventually, then $B$ is called a \emph{zero forcing set}.  The \emph{zero forcing number} is defined as the minimum cardinality of a zero forcing set.

Different types of zero forcing numbers have been discussed in the literature (e.g., see \cite{param,cancun}).  As we will see in Section~\ref{mrboundssec}, many of them serve as upper bounds for variants of the maximum nullity.  Here we recall three types of the zero forcing numbers $Z(G)$, $\Zelld(G)$, $\Zminus(G)$ from \cite{AIM,cancun,IMAISU} and introduce a new parameter, the \emph{$L$-zero forcing number} $\ZL(G)$.  

\begin{defn}
\label{Zdefn}
On a graph where vertices are colored blue or white, the color change rule for each of $Z$, $\Zelld$, $\Zminus$, and $\ZL$ are as follows.
\begin{enumarabic}
\item (CCR-$Z$) If $y\in N(x)$ and $N[x]$ are all blue except for
  $y$, then $y$ turns blue.  Denoted as $x\Zarrow y$.
\item (CCR-$\Zelld$) If $y\in N[x]$ and $N[x]$ are all blue except for
  $y$, then $y$ turns blue.  Denoted as $x\Zelldarrow y$.
\item (CCR-$\Zminus$) If $y\in N(x)$ and $N(x)$ are all blue except for
  $y$, then $y$ turns blue.  Denoted as $x\Zminusarrow y$.
\item (CCR-$\ZL$) Either $x\Zminusarrow y$ when $x\neq y$, or $x\Zelldarrow y$ when $x=y$ may apply.  Denoted as $x\ZLarrow y$.
\end{enumarabic}
\end{defn}

One may think of CCR-$\ZL$ as the following:  If $y\in N[x]$ and $N(x)$ are all blue ``except'' for $y$, then $y$ turns blue.  However, this definition is not clear in the case of $x=y$, so we explicitly separate it into two cases: $x\Zminusarrow y$ when $x\neq y$, or $x\Zelldarrow y$ when $x=y$. 

\begin{rem}
In \cite{cancun}, the zero forcing number of a loop graph is defined.  The parameter $\Zelld(G)$ is the same as the zero forcing number of a loop graph that is obtained from $G$ by considering every vertex as having a loop.  Notice that $\Zelld(G)$ is slightly different from the loop zero forcing number $\Zell(G)$ defined in \cite{param}.  When $G$ has exactly $r$ isolated vertices, $\Zell(G)=\Zelld(G)+r$.

Similarly, the parameter $\Zminus(G)$ is the same as the zero forcing number of a loop graph that is obtained from $G$ by considering every vertex as having no loop.  It is also called the skew zero forcing number in \cite{IMAISU} for studying the minimum rank problem on skew-symmetric matrices. 
\end{rem}

When the color change rule in effect is clear, we sometimes omit the superscript above the arrow and write $a\rightarrow b$ as a force.  In a zero forcing game with a given color change rule, the \emph{chronological list} records the performed forces $a_i\rightarrow b_i$ in the chronological order.  A game is called \emph{successful} if all vertices turn blue at the end.  If $(a_i\rightarrow b_i)_{i=1}^k$ is the chronological list of a successful zero forcing game, then $b_1,\ldots ,b_k$ are the initial white vertices and $V(G)\setminus\{b_1,\ldots, b_k\}$ is the set of initial blue vertices, which is a zero forcing set.

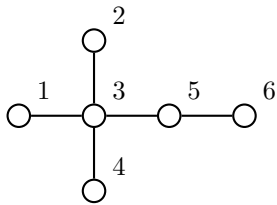
\begin{figure}[h]
\begin{center}
\begin{tikzpicture}
\node (1) [label={45:1}] at (0,1) {};
\node (2) [label={45:2}] at (1,2) {};
\node (3) [label={45:3}] at (1,1) {};
\node (4) [label={45:4}] at (1,0) {};
\node (5) [label={45:5}] at (2,1) {};
\node (6) [label={45:6}] at (3,1) {};
\draw (1) -- (3) -- (5) -- (6);
\draw (2) -- (3) -- (4);
\end{tikzpicture}
\end{center}
\caption{An example where $Z$, $\Zminus$, $\Zelld$, and $\ZL$ are all distinct}
\label{allZdiff}
\end{figure}

%%%Check the typesetting
\begin{ex}
\label{exZ}
Let $G$ be the graph as in Figure \ref{allZdiff}.  Then $Z(G)=3$, $\Zelld(G)=1$, $\Zminus(G)=2$, and $\ZL(G)=0$.  For each zero forcing type parameters, an minimum zero forcing set $B$ along with its chronological list is shown below.
\vbox{
\begin{multicols}{4}
\noindent
\[\begin{array}{c}
Z(G)=3 \\
B=\{1,2,4\} \\
\\
1\Zarrow 3 \\
3\Zarrow 5 \\
5\Zarrow 6 \\
\end{array}\]
\vfill\null\columnbreak
\noindent
\[\begin{array}{c}
\Zelld(G)=1 \\
B=\{3\} \\
\\
1\Zelldarrow 1 \\
2\Zelldarrow 2 \\
4\Zelldarrow 4 \\
3\Zelldarrow 5 \\
5\Zelldarrow6 \\
\end{array}\]
\vfill\null\columnbreak
\noindent
\[\begin{array}{c}
\Zminus(G)=2 \\
B=\{1,2\} \\
\\
4\Zminusarrow 3 \\
5\Zminusarrow 6 \\
6\Zminusarrow 5 \\
3\Zminusarrow 4 \\
\end{array}\]
\vfill\null\columnbreak
\noindent
\[\begin{array}{c}
\ZL(G)=0 \\
B=\{\} \\
\\
1\ZLarrow 3 \\
1\ZLarrow 1 \\
2\ZLarrow 2 \\
4\ZLarrow 4 \\
3\ZLarrow 5 \\
5\ZLarrow 6 \\
\end{array}\]
\end{multicols}}
\end{ex}

\begin{defn}
Let $G$ be a graph. A sequence $(v_1, v_2,\ldots, v_k)$ of distinct vertices of $G$ is called
\begin{enumarabic} 
\item a \emph{$Z$-sequence} if $N(v_i)\setminus\bigcup_{j=1}^{i-1}N[v_j]\neq\emptyset$ for all $i=1,\ldots,k$;
\item a \emph{dominating sequence} if $N[v_i]\setminus\bigcup_{j=1}^{i-1}N[v_j]\neq\emptyset$ for all $i=1,\ldots,k$;
\item a \emph{total dominating sequence} if $N(v_i)\setminus\bigcup_{j=1}^{i-1}N(v_j)\neq\emptyset$ for all $i=1,\ldots,k$;
\item an \emph{$L$-sequence} if $N[v_i]\setminus\bigcup_{j=1}^{i-1}N(v_j)\neq\emptyset$ for all $i=1,\ldots,k$.
\end{enumarabic}
The \emph{$Z$-Grundy domination number} $\gdZ(G)$ is the largest length of a $Z$-sequence;  the \emph{Grundy domination number} $\gd(G)$ is the largest length of a dominating sequence;  the \emph{Grundy total domination number} $\gdt(G)$ is the largest length of a total dominating sequence;  the \emph{$L$-Grundy domination number} $\gdL(G)$ is the largest length of a $L$-sequence.  
\end{defn}

\begin{ex}
Let $G$ be the graph in Figure \ref{allZdiff}.  Then $(6,5,3)$ is a $Z$-sequence.  Indeed, for each chronological list shown in Example \ref{exZ}, reading the second column from the bottom to the top gives the corresponding sequence.  That is, $(6,5,4,2,1)$ is a dominating sequence, $(4,5,6,3)$ is a total dominating sequence, and $(6,5,4,2,1,3)$ is an $L$-sequence.  
\end{ex}

\begin{rem}
In \cite{BBGKKPTV17}, the $Z$-Grundy domination number and the Grundy total domination number are defined on graphs without isolated vertices.  However, the same definitions are well-defined for all graph.  Indeed, if a graph $G$ can be written as $H\dunion rK_1$ such that $H$ has no isolated vertices, then $\gdZ(G)=\gdZ(H)$ and $\gdt(G)=\gdt(H)$.
\end{rem}

\section{The four zero forcing type parameters and the four Grundy domination type parameters}
\label{mainsec}

In this section we give the relations between the four zero forcing type parameters $Z$, $\Zelld$, $\Zminus$, $\ZL$ and the four Grundy domination type parameters $\gdZ$, $\gd$, $\gdt$, $\gdL$.  Note that the identity $Z(G)+\gdZ(G)=|V(G)|$ has been shown in \cite{BBGKKPTV17}.

\begin{lem}
\label{mainlemma}
Let $G$ be a graph and $(b_1,b_2,\ldots ,b_k)$ a sequence of vertices of $G$.  There exist vertices $a_1,a_2,\ldots,a_k$ such that $(a_i\rightarrow b_i)_{i=1}^k$ is the chronological list of a successful zero forcing game with 
\begin{multicols}{2}
\begin{enumarabic}
\item CCR-$Z$, 
\item CCR-$\Zelld$,
\item CCR-$\Zminus$, or 
\item CCR-$\ZL$
\end{enumarabic}
\end{multicols}
 if and only if $(b_k,b_{k-1},\ldots, b_1)$ is 
\begin{multicols}{2}
\begin{enumarabic}
\item a $Z$-sequence,
\item a dominating sequence,
\item a total dominating sequence, or 
\item a $L$-sequence,
\end{enumarabic}
\end{multicols}
respectively.
\end{lem}
\begin{proof}
Case (1) is in \cite{BBGKKPTV17}.  Here we prove Case (4), as the others are similar.

Suppose there are vertices $a_1,a_2,\ldots ,a_k$ such that $(a_i\ZLarrow b_i)_{i=1}^k$ is the chronological list of a successful zero forcing game with CCR-$\ZL$.  We will claim that 
\[a_i\in N[b_i]\setminus\bigcup_{j=i+1}^k N(b_j)\]
for all $i$.  Consequently, $(b_k,b_{k-1},\ldots, b_1)$ is an $L$-sequence.  For each $i=1,\ldots, k$, it must be one of the two cases following.

{\bf Case A:} $a_i=b_i$ and $a_i\ZLarrow b_i$ is done through $a_i\Zelldarrow b_i$.  This means by the time $a_i\ZLarrow b_i$ all vertices in $N(a_i)$ are blue and $b_i$ is white.  At this moment, the white vertices are $b_i, b_{i+1},\ldots,b_k$, so $b_j\notin N(a_i)$ for all $j=i+1,\ldots ,k$, and $a_i\notin \bigcup_{j=i+1}^k N(b_j)$.  Also, $a_i=b_i\in N[b_i]$.

{\bf Case B:} $b_i\in N(a_i)$ and $a_i \ZLarrow b_i$ is done through $a_i \Zminusarrow b_i$.  This means by the time $a_i\ZLarrow b_i$ all vertices in $N(a_i)$ are blue except for $b_i$.  At this moment, the white vertices are $b_i, b_{i+1},\ldots,b_k$, so $b_j\notin N(a_i)$ for all $j=i+1,\ldots ,k$, and $a_i\notin \bigcup_{j=i+1}^k N(b_j)$.

Conversely, suppose $(b_k,b_{k-1},\ldots,b_1)$ is an $L$-sequence.  Pick $a_i$ as an element in $N[b_i]\setminus\bigcup_{j=i+1}^k N(b_j)$ for each $i$.  We will show that $(a_i\ZLarrow b_i)_{i=1}^k$ is the chronological list of a successful zero forcing game with CCR-$\ZL$, starting with the initial blue set 
\[V(G)\setminus\{b_1,b_2,\ldots,b_k\}.\]
To see this, assume at the $i$-th step the vertices $b_1,\ldots, b_{i-1}$ are blue, or equivalently, the white vertices are $b_i,\ldots, b_k$.

{\bf Case A:} $a_i=b_i$ and $a_i\notin \bigcup_{j=i+1}^k N(b_j)$.  This means $b_j\notin N(a_i)$ for all $j=i+1,\ldots,k$, and $N(a_i)$ are all blue.  Therefore, $a_i\Zelldarrow b_i$ applies.

{\bf Case B:} $a_i\in N(b_i)$ and $a_i\notin \bigcup_{j=i+1}^k N(b_j)$.  This means $b_j\notin N(a_i)$ for all $j=i+1,\ldots,k$, and $N(a_i)$ are all blue except for $b_i$.  Thus, $b_i$ is the only white vertex in $N(a_i)$, and $a_i\Zminusarrow b_i$ applies.
\end{proof}

\begin{thm}
\label{mainthm}
Let $G$ be a graph and $|V(G)|=n$.  Then 
\begin{multicols}{2}
\begin{enumarabic}
\item $Z(G)=n-\gdZ(G)$,
\item $\Zelld(G)=n-\gd(G)$,
\item $\Zminus(G)=n-\gdt(G)$,
\item $\ZL(G)=n-\gdL(G)$.
\end{enumarabic}
\end{multicols}
\end{thm}
\begin{proof}
This follows immediately from Lemma~\ref{mainlemma}.
\end{proof}

\begin{prop}
Let $G$ be a graph on $n$ vertices.  Then 
\begin{multicols}{2}
\begin{enumarabic}
\item $\ZL(G)\leq \Zelld(G)\leq Z(G)$, 
\item $\ZL(G)\leq \Zminus(G)\leq Z(G)$,
\item $2Z(G)\leq n+\Zminus(G)$, and 
\item $2\Zelld(G)\leq n+\ZL(G)$.
\end{enumarabic}
\end{multicols}
Moreover, $\ZL(G)+1\leq \Zelld(G)$ if $G$ has at least an edge.
\end{prop}
\begin{proof}
For (1) and (2), the inequalities follow from the definition.

Suppose $G$ has exactly $r$ isolated vertices and is written as $H\dunion rK_1$, where $H$ does not have any vertex if $G$ has no edge.

If $G$ has at least an edge, then $H$ is a (non-degenerated) graph without isolated vertices.  By \cite{BBGKKPTV17}, $\gd(H)\leq \gdL(H)-1$, so 
\[\begin{aligned}
\ZL(G)+1&=n-\gdL(G)+1=n-\gdL(H)-r+1 \\
 &\leq n-\gd(H)-r=n-\gd(G)=\Zelld(G)
\end{aligned}\]
by Theorem~\ref{mainthm}.

For (3), it is known \cite{BBGKKPTV17} that $\gdt(H)\leq 2\gdZ(H)$ for any graph without isolated vertices, so 
\[\begin{aligned} 
2Z(G) &=2n-2\gdZ(G)=2n-2\gdZ(H) \\
 &\leq 2n-\gdt(H)=2n-\gdt(G)=n+\Zminus(G)
\end{aligned}\]
by Theorem~\ref{mainthm}.

For (4), it is known \cite{BBGKKPTV17} that $\gdL(G)\leq 2\gd(G)$ for any graph, so 
\[2\Zelld(G) =2n-2\gd(G)\leq 2n-\gdL(G)=n+\ZL(G)\]
by Theorem~\ref{mainthm}.
\end{proof}

Recall that $\gamma(G)$ is the domination number.  Since $\gd(G)$ considers the worst case in the greedy algorithm of finding $\gamma(G)$, we know $\gamma(G)\leq \gd(G)$.  Indeed, for a graph $G$ without isolated vertices, any maximal $Z$-sequence also gives a dominating set, so $\gamma(G)\leq\gdZ(G)$.  This gives an upper bound to $Z(G)$.

\begin{prop}
Let $G$ be a graph without isolated vertex and $n=|V(G)|$.  Then $Z(G)\leq n-\gamma(G)$.
\end{prop}
%\begin{proof}
%Since $G$ has no isolated vertex, every maximal $Z$-sequence is a dominating set, implying $\gamma(G)\leq\gdZ(G)$.  Then $Z(G)\leq n-\gamma(G)$ by Theorem~\ref{mainthm}.
%\end{proof}
%It follows immediately from Theorem~\ref{mainthm}.  Here we include a proof regarding the definition of the zero forcing number.  Suppose $B$ is a minimal zero forcing set.  Then every vertex $u\in B$ has a neighbor which is not in $B$, for otherwise $B\setminus\{v\}$ is a smaller zero forcing set for any $v\in N(u)$.  This is because starting with $B\setminus\{v\}$, we may perform $u\Zarrow v$ so that $v$ turns blue and the remaining process may continue.  Equivalently, this means $V(G)\setminus B$ is a dominating set. 

\section{The minimum rank type parameters}
\label{mrboundssec}

Let $G$ be a graph on $n$ vertices.  Define $\S(G)$ as the family of $n\times n$ real symmetric matrices whose $i,j$-entry, $i\neq j$, is nonzero whenever $\{i,j\}\in E(G)$ and zero otherwise.  Notice that there are no restrictions on the diagonal entries.  The \emph{maximum nullity} and the \emph{minimum rank} of $G$ are defined as
\[
\begin{aligned}
M(G) &= \max\{\nul(A): A\in\S(G)\} \text{ and}\\
\mr(G) &= \min\{\rank(A): A\in\S(G)\},
\end{aligned}
\]
respectively.  By definition, $M(G)+\mr(G)=|V(G)|$.  It is shown in \cite{AIM} that $M(G)\leq Z(G)$ for every graph.

Among matrices in $\S(G)$, let $\Selld(G)$ be the matrices with every diagonal entry nonzero, and let $\Szero(G)$ be the matrices with every diagonal entry zero.  Similarly, define
\[
\begin{aligned}
\Melld(G) &= \max\{\nul(A): A\in\Selld(G)\},\\
\mrelld(G) &= \min\{\rank(A): A\in\Selld(G)\},\\
\Mzero(G) &= \max\{\nul(A): A\in\Szero(G)\},\\
\mrzero(G) &= \min\{\rank(A): A\in\Szero(G)\}.
\end{aligned}
\]
By definition, $\Melld(G)+\mrelld(G)=|V(G)|$ and $\Mzero(G)+\mrzero(G)=|V(G)|$.  Also, it is known \cite{cancun} that $\Melld(G)\leq \Zelld(G)$ and $\Mzero(G)\leq \Zminus(G)$.  

Usually, finding a lower bound for a Grundy domination type parameter is by constructing a sequence and verify if the sequence has the desired property.  On the other side, finding an upper bound for a Grundy domination type parameter requires an argument showing every sequence with the corresponding properties cannot be too long.  Theorem~\ref{mrbounds} gives a fairly easy way to find upper bounds for the Grundy domination type parameters.  (The upper bound for the $L$-Grundy domination number will be provided in Section~\ref{lgboundsec}.)

\begin{thm}
\label{mrbounds}
Let $G$ be a graph.  Then 
\begin{enumarabic}
\item $\gdZ(G)\leq\mr(G)$,
\item $\gd(G)\leq\mrelld(G)$, and 
\item $\gdt(G)\leq\mrzero(G)$.
\end{enumarabic}
\end{thm}
\begin{proof}
By \cite{AIM, cancun}, $M(G)\leq Z(G)$, $\Melld(G)\leq\Zelld(G)$, and $\Mzero(G)\leq\Zminus(G)$.  This is equivalent to $\gdZ(G)\leq\mr(G)$, $\gd(G)\leq\mrelld(G)$, and $\gdt(G)\leq\mrzero(G)$ by Theorem~\ref{mainthm}. 
\end{proof}

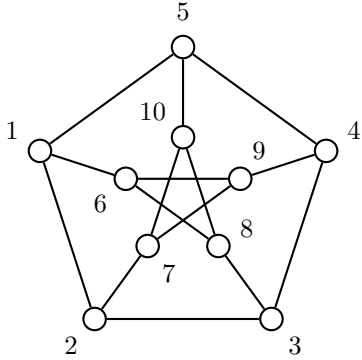
\begin{figure}[h]
\begin{center}
\begin{tikzpicture}
\foreach \i in {1,...,5}{
\pgfmathsetmacro{\j}{int(\i+5)}
\pgfmathsetmacro{\ang}{90+(72*\i)}
\pgfmathsetmacro{\lang}{\ang+60}
\node[label={\ang:$\i$}] (\i) at (\ang:2) {};
\node[label={\lang:$\j$}] (\j) at (\ang:0.8) {};
\draw (\i) -- (\j);
}
\draw (1) -- (2) -- (3) -- (4) -- (5) -- (1);
\draw (6) -- (8) -- (10) -- (7) -- (9) -- (6);
\end{tikzpicture}
\end{center}
\caption{A labeled Petersen graph}
\label{Petersen}
\end{figure}

\begin{ex}
Let $P$ be the Petersen graph as drawn in Figure~\ref{Petersen}.  Then $(1,2,3,4,5)$ is a $Z$-sequence and also a dominating sequence.  And $(9,1,2,3,4,5)$ is a total dominating sequence.

Let
\[C=\begin{bmatrix}
0&1&0&0&1\\
1&0&1&0&0\\
0&1&0&1&0\\
0&0&1&0&1\\
1&0&0&1&0\\
\end{bmatrix}\text{ and }
C'=\begin{bmatrix}
0&0&1&1&0\\
0&0&0&1&1\\
1&0&0&0&1\\
1&1&0&0&0\\
0&1&1&0&0\\
\end{bmatrix}.\]
Let 
\[A=\begin{bmatrix}
C-I_5 & I_5 \\
I_5 & C'-I_5 \\
\end{bmatrix}\text{ and }
B=\begin{bmatrix}
-C & I_5 \\
I_5 & C' \\
\end{bmatrix},\]
where $I_5$ is the identity matrix of order $5$. 
One may check that $A\in \Selld(G)\subseteq \S(G)$ and $\rank(A)=5$; also, $B\in\Szero(G)$ and $\rank(B)=6$.  Therefore, we know $\gdZ(G)=\gd(G)=5$ and $\gdt(G)=6$.
\end{ex}

Recall that the \emph{independence number} $\alpha(G)$ is the largest cardinality of a independent set, and the \emph{vertex cover number} $\beta(G)$ is the minimum number of vertices such that every edge is incident to at least one of these vertices.  It is a standard result that $\alpha(G)+\beta(G)=|V(G)|$.  In \cite{TDinGraph}, it was shown that $\gdt(G)\leq 2\beta(G)$.  Here we improve this result by showing $\mrzero(G)\leq 2\beta(G)$. 

\begin{prop}
Let $G$ be a graph.  Then $\Zminus(G)\geq \Mzero(G)\geq\alpha(G)-\beta(G)$ and $\gdt(G)\leq \mrzero(G)\leq 2\beta(G)$.
\end{prop}
\begin{proof}
Let $A$ be a matrix in $\Szero(G)$ and $X$ an independent set of $G$ with $|X|=\alpha(G)$.  Since there is no edges in $X$, the submatrix $A[X]$ of $A$ induced on rows and columns in $X$ is a zero matrix, and $\rank(A[X])=0$.  Also, the matrix $A$ can be obtained from $A[X]$ by adding $\beta(G)$ rows and $\beta(G)$ columns, since $|V(G)|-\alpha(G)=\beta(G)$.  Adding a row or a column can increase the rank by at most one, so 
\[\rank(A)\leq \rank(A[X])+2\beta(G)=2\beta(G)\]
and $\mrzero(G)\leq \rank(A)\leq 2\beta(G)$.  Other inequalities follows from Theorem~\ref{mainthm} and Theorem~\ref{mrbounds}.
\end{proof}

The \emph{edge clique cover number} $\cc(G)$ is the minimum number of cliques that can cover every edge of $G$.  It is noted in \cite{BBGKKPTV17} that $\gd(G)\leq \cc(G)$.  Here we prove that $\mrelld(G)\leq \cc(G)$; the technique is standard in the field of the minimum rank problem.

\begin{prop}
Let $G$ be a graph.  Then $\Zelld(G)\geq \Melld(G)\geq n-\cc(G)$ and 
$\gd(G)\leq \mrelld(G)\leq \cc(G)$.
\end{prop}
\begin{proof}
Let $k=\cc(G)$ and $\{C_1,\ldots,C_k\}$ an edge clique cover.  For each clique $C_t$ we may construct a rank-one matrix $A_t$ where the $i,j$-entry is $1$ whenever both $i$ and $j$ ($i=j$ is possible) are in $C_t$ and zero otherwise.  Thus, the matrix $A=\sum_{t=1}^k A_t$ is a matrix in $\Selld(G)$ with $\rank(A)\leq k$.  Therefore, $\mrelld(G)\leq k=\cc(G)$.  Other inequalities follows from Theorem~\ref{mainthm} and Theorem~\ref{mrbounds}.
\end{proof}

%investigate
%\[\operatorname{ir}(G)\leq\gamma(G)\leq \gamma_i\leq \alpha(G)\leq \Gamma(G)\leq \operatorname{IR}(G)\leq\gd(G).\]

\section{Linear algebra bounds for the $L$-Grundy domination number}
\label{lgboundsec}

In this section, we will provide a relation between the $L$-Grundy domination number and the Grundy total domination number.  Then we use this relation to provide some linear algebra bounds for the $L$-Grundy domination number.

Let $G$ be a graph on the vertex set $\{1,\ldots,n\}$.  Construct a bipartite graph $\BL(G)$ with 
\begin{itemize}
\item $V(\BL(G))=\{x_1,\ldots,x_n\}\cup\{y_1,\ldots,y_n\}\cup\{z_1,\ldots,z_n\}$,
%X\cup Y\cup Z$ with $X=\{x_1,\ldots,x_n\}$, $Y=\{y_1,\ldots,y_n\}$, and $Z=\{z_1,\ldots,z_n\}$, 
\item $E(\BL(G))=\{\{x_i,y_j\},\{x_i,z_j\}:\{i,j\}\in E(G)\}\cup\{\{x_i,y_i\}:i\in V(G)\}$.
\end{itemize}

Let $G$ be a graph and $X\subseteq V(G)$.  We say $\gdt(G,X)$ is the maximum length of a total dominating sequence of $G$ using only vertices in $X$.  

\begin{figure}[h]
\begin{center}
\begin{tikzpicture}
\foreach \i in {1,...,4}{
\node[label={45:\i}] (\i) at (0,-\i) {};
}
\draw (1) -- (2) -- (3) -- (4);
\node[rectangle,draw=none] at (0,-4.5) {$G$};
\end{tikzpicture}
\hfil
\begin{tikzpicture}
\foreach \i in {1,...,4}{
\node[label={[label distance=-4pt]above:$x_{\i}$}] (a\i) at (0,-\i) {};
\node[label={135:$y_{\i}$}] (b\i) at (-1,-\i) {};
\node[label={45:$z_{\i}$}] (c\i) at (1,-\i) {};
\draw (a\i)--(b\i);
}
\foreach \i/\j in {1/2,2/3,3/4}{
\draw (a\i)--(b\j);
\draw (b\i)--(a\j);
\draw (a\i)--(c\j);
\draw (c\i)--(a\j);
}
\node[rectangle,draw=none] at (0,-4.5) {$\BL(G)$};
\end{tikzpicture}
\end{center}
\caption{An illustration of $G$ and $\BL(G)$}
\label{BLexample}
\end{figure}
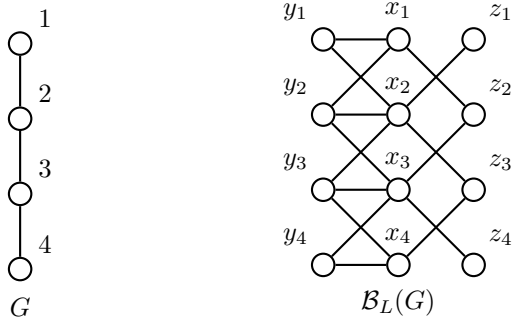

\begin{ex}
Let $G$ be the path on $4$ vertices.  Then the graph $\BL(G)$ is as shown in Figure \ref{BLexample}.  The graph $G$ has $\gdL(G)=4$ and $\ZL(G)=0$, as $(4,3,1,2)$ is a maximum $L$-sequence in $G$.  
%\[\begin{array}{c}
%1\ZLarrow 2\\
%1\ZLarrow 1\\
%2\ZLarrow 3\\
%3\ZLarrow 4\\
%\end{array}\]

At the same time, $(x_4,x_3,x_1,x_2)$ is a total dominating sequence in $\BL(G)$, so $\gdt(G,X)=4$ for $X=\{x_1,x_2,x_3,x_4\}$.
\end{ex}

When $G$ is a bipartite graph, its total dominating sequence has a nice decomposition, as shown in Proposition~8.3 and Theorem~8.4 of \cite{TDinGraph}.

\begin{prop}
\label{gdtbipartite}
Let $G$ be a bipartite graph with two parts $X$ and $Y$.  Then $\gdt(G)=2\gdt(G,X)=2\gdt(G,Y)$.
\end{prop}
\begin{proof}
Any bipartite graph $G$ can be the incidence graph of a hypergraph $\mathcal{H}$; that is, $X$ represents the vertices of $\mathcal{H}$, $Y$ represents the edges of $\mathcal{H}$, and there is an edge between $x\in X$ and $y\in Y$ if and only if vertex $x$ is incident to edge $y$ in $\mathcal{H}$.  Following the notation in \cite{TDinGraph}, we have $\gdt(G,X)=\tau_{\rm gd}(\mathcal{H})$ and $\gdt(G,Y)=\rho_{\rm gd}(\mathcal{H})$, so the desired results follow from Proposition~8.3 and Theorem~8.4 of \cite{TDinGraph}.
\end{proof}

\begin{thm}
\label{ZLidentity}
Let $G$ be a graph on $n$ vertices.  Then $\Zminus(\BL(G))=n+2\ZL(G)$ and $\gdt(\BL(G))=2\gdL(G)$.
\end{thm}
\begin{proof}
Let $X=\{x_1,\ldots,x_n\}$.  We will show that $\gdt(\BL(G),X)=\gdL(G)$.  Then the desired results follow form Proposition~\ref{gdtbipartite} and Theorem~\ref{mainthm}.

Let $k=\gdL(G)$ and $(v_1,\ldots, v_k)$ an $L$-sequence of $G$.  We can verify that $(x_{v_1},\ldots, x_{v_k})$ is a total dominating sequence of $\BL(G)$.  Since $(v_1,\ldots,v_k)$ is an $L$-sequence in $G$, we may pick  
\[u_i\in N_G[v_i]\setminus\bigcup_{j=1}^{i-1}N_G(v_j)\]
for each $i$ with $1\leq i\leq k$.  If $u_i\neq v_i$, then 
\[z_{u_i}\in N_{\BL(G)}(x_{v_i})\setminus\bigcup_{i=1}^{i-1}N_{\BL(G)}(x_{v_j}).\]
If $u_i=v_i$, then 
\[y_{u_i}\in N_{\BL(G)}(x_{v_i})\setminus\bigcup_{i=1}^{i-1}N_{\BL(G)}(x_{v_j}).\]
Therefore, $(x_{v_1},\ldots,x_{v_k})$ is a total dominating sequence of $\BL(G)$, and 
\[\gdt(\BL(G),X)\geq k=\gdL(G).\]

Conversely, suppose $h=\gdt(\BL(G),X)$ and $(x_{v_1},\ldots,x_{v_h})$ is a total dominating sequence of $\BL(G)$.  We can verify that $(v_1,\ldots,v_h)$ is an $L$-sequence of $G$.  Since $(x_{v_1},\ldots,x_{v_h})$ is a total dominating sequence of $\BL(G)$, we may pick an element in 
\[N_{\BL(G)}(x_{v_i})\setminus\bigcup_{i=1}^{i-1}N_{\BL(G)}(x_{v_j})\]
and let $u_i$ be its index.  Thus 
\[u_i\in N_G[v_i]\setminus\bigcup_{j=1}^{i-1}N_G(v_j).\]
Therefore, $(v_1,\ldots,v_h)$ is an $L$-sequence of $G$, and $\gdL(G)\geq h=\gdt(\BL(G),X)$.  This completes the proof.
\end{proof}

\begin{cor}
\label{asymbound}
Let $G$ be a graph on $n$ vertices.  Then 
\[\frac{\Mzero(\BL(G))-n}{2}\leq \ZL(G)\text{ and }\gdL(G)\leq \frac{1}{2}\mrzero(\BL(G)).\]
\end{cor}
\begin{proof}
By Theorem~\ref{ZLidentity} and Theorem~\ref{mrbounds}, 
\[\gdL(G)=\frac{1}{2}\gdt(\BL(G))\leq\frac{1}{2}\mrzero(\BL(G)).\]  
For the other inequality follows, 
\[\begin{aligned} \ZL(G) & =n-\gdL(G)\geq n-\frac{1}{2}\mrzero(\BL(G)) \\
 & = n-\frac{1}{2}(3n-\Mzero(\BL(G)))=\frac{\Mzero(\BL(G))-n}{2}
\end{aligned}\]
by Theorem~\ref{mainthm}.
\end{proof}

For a graph $G$ on $n$ vertices, define $\mrL(G)$ as the minimum rank over matrices of the form $\begin{bmatrix} A & B \end{bmatrix}$ with $A\in\Selld(G)$ and $B\in\Szero(G)$, and let $\ML(G)=n-\mrL(G)$.

\begin{cor}
For any graph $G$, $\ML(G)\leq \ZL(G)$ and $\gdL(G)\leq \mrL(G)$.
\end{cor}
\begin{proof}
Let $A\in\Selld(G)$ and $B\in\Szero(G)$ be matrices such that $\rank(\begin{bmatrix} A & B \end{bmatrix})=\mrL(G)$.  Then the matrix 
\[\begin{bmatrix}
O & A & B \\
A & O & O \\
B & O & O \\
\end{bmatrix}\]
is a matrix in $\Szero(\BL(G))$ with rank $2\mrL(G)$.  Therefore, by Corollary~\ref{asymbound}
\[\gdL(G)\leq \frac{1}{2}\mrzero(\BL(G))\leq \mrL(G).\]
The other inequality follows from Theorem~\ref{mainthm}. 
\end{proof}

\begin{figure}[h]
\begin{center}
\begin{tikzpicture}
\foreach \i in {1,2,3}{
\pgfmathsetmacro{\j}{int(\i+3)}
\node [label={left:$\i$}] (a\i) at (0,\i) {};
\node [label={right:$\j$}] (b\i) at (1.5,\i) {};
}
\foreach \i in {1,2,3}{
\foreach \j in {1,2,3}{
\draw (a\i) -- (b\j);
}
}
\end{tikzpicture}
\end{center}
\caption{A labeled $K_{3,3}$}
\label{K33}
\end{figure}
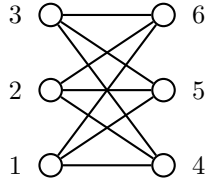

\begin{ex}
Let $G$ be the complete bipartite graph $K_{3,3}$ as shown in Figure~\ref{K33}.  Then $(1,2,3,4)$ is an $L$-sequence.  The rectangular matrix 
\[C=\begin{bmatrix}
3 & 0 & 0 & 1 & 1 & -2 & 0 & 0 & 0 &
1 & 1 & 1 \\
0 & 3 & 0 & 1 & -2 & 1 & 0 & 0 & 0 &
1 & 1 & 1 \\
0 & 0 & 3 & -2 & 1 & 1 & 0 & 0 & 0 &
1 & 1 & 1 \\
1 & 1 & -2 & 3 & 0 & 0 & 1 & 1 & 1 &
0 & 0 & 0 \\
1 & -2 & 1 & 0 & 3 & 0 & 1 & 1 & 1 &
0 & 0 & 0 \\
-2 & 1 & 1 & 0 & 0 & 3 & 1 & 1 & 1 &
0 & 0 & 0 \\
\end{bmatrix}\]
is of the form $\begin{bmatrix} A & B \end{bmatrix}$ with $A\in\Selld(G)$ and $B\in\Szero(G)$.  Since $\rank(C)=4$, the sequence $(1,2,4,5)$ is a maximum $L$-sequence.
\end{ex}

\section{Conclusion}  

In this paper we found the relations between the four zero forcing type parameters $Z$, $\Zelld$, $\Zminus$, $\ZL$ and the four Grundy domination type parameters $\gdZ$, $\gd$, $\gdt$, $\gdL$.  With these relation we are allowed to bring results from one sides to the other; for example, Theorem~\ref{mrbounds} provides upper bounds of the Grundy domination type parameters due to the application of the zero forcing type parameters to the minimum rank problem.

There are yet more zero forcing type parameters; see, e.g., \cite{param, Zsap, Zoc}.  In \cite{smallparam}, the positive semi-definite zero forcing number $Z_+(G)$ was shown to have the relation $Z_+(G)+\OS(G)=|V(G)|$.  Here $\OS(G)$ is the \emph{OS-number} introduced in  \cite{HHLMNP09} with its motivation from linear algebra and defined as the largest length of a sequence $(v_1,\ldots, v_k)$ such that 
\[N(v_i)\setminus\bigcup_{\substack{x\in V(H_i)\\x\neq v_i}}N[x],\]
where $H_i$ is the connected component of the induced subgraph $G[\{v_1,\ldots, v_i\}]$ that contains $v_i$.  With this definition, one may view the OS-number as an Grundy domination type parameters, and it would be interesting to see if there are further application of the OS-number to the domination problem.

On the other side, Bre\v{s}ar, Klav\v{z}ar, and Rall \cite{BKR10} introduced the domination game as follows.  Two players, Dominator and Staller, are picking vertices alternatively; each one needs to pick a vertex that dominate at least a new vertex that is not dominated by any previous chosen vertices.  Dominator wants to dominate all vertices as fast as possible, while Staller is trying to slow down the process.  The \emph{game domination number} $\gag(G)$ is the number of required steps to finish the game if starting with Dominator.  The game domination number and its variants should also have their counterparts as zero forcing type parameters, and it would be nice to see their applications to the minimum rank problem.

%\section{Acknowledgments}
%\nocite{*}

%\bibliography{./JLaTeX/AuthorA,./JLaTeX/JournalA,./JLaTeX/JepBib}{}
%\bibliographystyle{plain}

%bibitem code
\newcommand{\noopsort}[1]{}

\end{document}